\documentclass{article}

\usepackage{makeidx}
\usepackage{graphicx}
\usepackage{multicol}
\usepackage[bottom]{footmisc}
\usepackage{amsmath,amssymb,amsfonts, amsthm}
\usepackage[affil-it]{authblk}

\makeindex

\newcommand{\Z}{\mathbb Z}

\newcommand{\su}{\subseteq}

\newtheorem{thm}{Theorem}[section]
\newtheorem{lem}[thm]{Lemma}
\newtheorem{cor}[thm]{Corollary}
\newtheorem{obs}[thm]{Observation}

\newtheorem{rem}[thm]{Remark}

\theoremstyle{definition}
\newtheorem{defn}[thm]{Definition}

\theoremstyle{remark}
\newtheorem{case}{Case}[thm]

\author{Santanu Mondal, Krishnendu Paul, Shameek Paul%
\thanks{E-mail addresses: \texttt{santanu.mondal.math18@gm.rkmvu.ac.in, krishnendu.p.math18@gm.rkmvu.ac.in, shameek.paul@rkmvu.ac.in}}}
\affil{\small Ramakrishna Mission Vivekananda Educational and Research Institute, Belur, Dist. Howrah, 711202, India}

\date{}

\begin{document}
\baselineskip=14.5pt

\title {On unit weighted zero-sum constants of $\Z_n$}

\maketitle

\begin{abstract}
Given $A\su\Z_n$, the constant $C_A(n)$ is defined to be the smallest natural number $k$ such that any sequence of $k$ elements in $\Z_n$ has an $A$-weighted zero-sum subsequence having consecutive terms. The value of $C_{U(n)}(n)$ is known when $n$ is odd. We give a different argument to determine the value of $C_{U(n)}(n)$ for any $n$. A $C$-extremal sequence for $U(n)$ is a sequence in $\Z_n$ whose length is $C_{U(n)}(n)-1$ and which does not have any $U(n)$-weighted zero-sum subsequence having consecutive terms. We characterize the $C$-extremal sequences for $U(n)$ when $n$ is a power of 2. For any $n$, we determine the value of $C_A(n)$ where $A$ is the set of all odd (or all even) elements of $\Z_n$ and also when $A=\{1,2,\ldots,r\}$ where $r<n$. 
\end{abstract}



\vspace{.5cm}

\section{Introduction}\label{0}

We denote the number of elements in a finite set $S$ by $|S|$. In this paper, $R$ will denote a ring with unity. The next three definitions are given in \cite{SKS}. 

\begin{defn}
Let $M$ be an $R$-module and $A\su R$. A subsequence $T$ of a sequence $S=(x_1,x_2,\ldots, x_{k})$ in $M$ is called an {\it $A$-weighted zero-sum subsequence} if the set $I=\{i:x_i\in T\}$ is non-empty and for each $i\in I$, there exist $a_i\in A$ such that $\sum_{i\in I} a_i x_i = 0$. 
\end{defn}

\begin{defn} 
Given a finite $R$-module $M$ and $A\su R$, the $A$-weighted Davenport constant $D_A(M)$ is the least positive integer $k$ such that any sequence in $M$ of length $k$ has an $A$-weighted zero-sum subsequence. 
\end{defn}

\begin{defn}\label{ca}
Given a finite $R$-module $M$ and $A\su R$, the constant $C_A(M)$ is the least positive integer $k$ such that any sequence in $M$ of length $k$ has an $A$-weighted zero-sum subsequence of consecutive terms. 
\end{defn}

The next constant has been studied when $M=R=\Z_n$. Some results are given in \cite{A}, \cite{AR}, \cite{G}, \cite{L} and \cite{X}. 

\begin{defn}
Given a finite $R$-module $M$ and $A\su R$, the $A$-weighted Gao constant $E_A(M)$ is the least positive integer $k$ such that any sequence in $M$ of length $k$ has an $A$-weighted zero-sum subsequence of length $|M|$. 
\end{defn}

\begin{rem}
We denote the above three constants by $D(M),C(M)$ and $E(M)$ when $A=\{1\}$. 
For a finite $R$-module $M$, we have $D_A(M)\leq C_A(M)\leq |M|$ (see Section 1 of \cite{SKS}). Also, it follows easily  that $E_A(M)\geq D_A(M)+|M|-1$. 
\end{rem}

When $A\su \Z_n$, we denote the constants $D_A(\Z_n)$, $E_A(\Z_n)$ and $C_A(\Z_n)$ by $D_A(n)$, $E_A(n)$ and $C_A(n)$ respectively. Yuan and Zeng have shown in \cite{YZ} that $E_A(n)=D_A(n)+n-1$.

\begin{defn}
Let $A\su \Z_n$. A sequence in $\Z_n$ of length $E_A(n)-1$ which does not have any $A$-weighted zero-sum subsequence of length $n$, is called an $E$-extremal sequence for $A$. 
A sequence in $\Z_n$ of length $C_A(n)-1$ which does not have any $A$-weighted zero-sum subsequence of consecutive terms, is called a $C$-extremal sequence for $A$.  
\end{defn}

We denote the ring $\Z/n\Z$ by $\Z_n$. Let $U(n)$ denote the group of units in $\Z_n$ and  $U(n)^k=\{\,x^k:x\in U(n)\,\}$. When  $n=p_1^{r_1}\ldots p_s^{r_s}$ where the $p_i$'s are distinct primes, we let $\Omega(n)=r_1+\cdots +r_s$. We use the notation $v_p(n)=r$ to mean $p^r\mid n$ and $p^{r+1}\nmid n$. For $a,b\in\Z$, we denote the set $\{k\in\Z:a\leq k\leq b\}$ by $[a,b]$.

In Corollary 4 of \cite{SKS} it has been shown that $C_{U(n)}(n)=2^{\Omega(n)}$ when $n$ is odd. In Theorem \ref{c2a} we show that $C(\Z_2^{^a})=2^a$ and as a consequence of this in Theorem \ref{cun} we show that for {\it any} $n$ we have $C_{U(n)}(n)=2^{\Omega(n)}$. Thus, we obtain a new proof of Corollary 4 of \cite{SKS} while also  generalizing it.

In Theorem \ref{d2a} we rederive the result that $D(\Z_2^{^a})=a+1$. This result is a special case of the main result of Olson in \cite{O}. We have included it here as our proof is very elementary. As a consequence of this we rederive the result that $D_{U(n)}=\Omega(n)+1$. This can also be arrived at by using the result of Yuan and Zeng in \cite{YZ}, as the value of $E_{U(n)}(n)$ has been obtained in \cite{G} and \cite{L}. 

We feel that it is of interest to derive the value of $D_{U(n)}(n)$ without using the value of $E_{U(n)}(n)$ since for other weight-sets $A\su\Z_n$, the values of $D_A(n)$ have been determined independently of the corresponding values of $E_A(n)$. Some of the other results in this article are the following. 

\begin{itemize}
\item
For a finite abelian group $G$ we have that $C(G)=|G|$. 

\item 
We characterize the $C$-extremal and $E$-extremal sequences for $U(n)$ when $n$ is a power of 2.

\item
Let $n$ be even, $v_2(n)=r$ and $m=n/2^r$. Let $B$ denote the set of all odd elements of $\Z_n$ and $\{m\}\su A\su B$. Then we have that $C_A(n)=2^r$.

\item 
Let $r\in [1,n-1]$ and $A=\{1,2,\ldots,r\}$. We show that $C_A(n)=\lceil n/r \rceil$.

\item If $p$ and $k$ are odd primes such that $p\equiv 1~(mod~k)$ and $p\not\equiv 1~(mod~k^2)$, then $D_{U(p)^k}(p)\leq k$. 
\end{itemize}

\section{Some general results}

\begin{thm}
Let $M$ be a finite $R$-module and $N$ be a  finite $R'$-module. Suppose $A\su R$ and $B\su R'$. Then $M\times N$ is a module over $R\times R'$ and we have 
$C_{A\times B}(M\times N)\geq C_A(M)\,C_B(N)$. 
\end{thm}

\begin{proof}
Let $C_A(M)=k+1$ and $C_B(N)=l+1$. Suppose $S_1=(x_1,\ldots,x_k)$ is a sequence in $M$ which does not have any $A$-weighted zero-sum subsequence of consecutive terms and $S_2=(y_1,\ldots,y_l)$ is a sequence in $N$ which does not have any $B$-weighted zero-sum subsequence of consecutive terms. Consider the sequence $S$ in $M\times N$ defined as  
$$\big((x_1,0),\ldots,(x_k,0),(0,y_1),(x_1,0),\ldots,(x_k,0),(0,y_2),(x_1,0),\ldots,(x_k,0),(0,y_3),$$
$$.\,.\,.\,.\,,(x_1,0),\ldots,(x_k,0),(0,y_{l-1}),(x_1,0),\ldots,(x_k,0),(0,y_l),(x_1,0),\ldots,(x_k,0)\big).$$
Suppose $S$ has an $A\times B$-weighted zero-sum subsequence $T$ having consecutive terms. If no term of $T$ is of the form $(0,y_j)$ for any $j\in [1,l]$, then we get the contradiction that $S_1$ has an $A$-weighted zero-sum subsequence of consecutive terms. If $T$ has a term of the form $(0,y_j)$ for some $j\in [1,l]$, then by taking the projection to the second coordinate, we get the contradiction that $S_2$ has a $B$-weighted zero-sum subsequence of consecutive terms. Thus, we see that $S$ does not have any $A\times B$-weighted zero-sum subsequence of consecutive terms. Hence, we get that $C_{A\times B}(M\times N)\geq C_A(M)\,C_B(N)$. 
\end{proof}

By a similar argument we get the next result. 

\begin{thm}\label{proc}
Let $M$ and $N$ be $R$-modules and $A\su R$. Then $M\times N$ is an $R$-module and we have $C_A(M\times N)\geq C_A(M)\,C_A(N)$. 
\end{thm}

\begin{thm}
Let $M$ be a finite $R$-module and $N$ be a  finite $R'$-module. Suppose $A\su R$ and $B\su R'$. Then $M\times N$ is a module over $R\times R'$ and we have 
$D_{A\times B}(M\times N)\geq D_A(M)+D_B(N)-1$. 
\end{thm}

\begin{proof}
Let $D_A(M)=k+1$ and $D_B(N)=l+1$. Suppose $S_1=(x_1,\ldots,x_k)$ is a sequence in $M$ which does not have any $A$-weighted zero-sum subsequence and $S_2=(y_1,\ldots,y_l)$ is a sequence in $N$ which does not have any $B$-weighted zero-sum subsequence. Let $S$ be the sequence in $M\times N$ defined as $$\big((x_1,0),(x_2,0),\ldots,(x_{k-1},0),(x_k,0),(0,y_1),(0,y_2),\ldots,(0,y_{l-1}),(0,y_l)\big).$$
Suppose $S$ has an $A\times B$-weighted zero-sum subsequence $T$. If no term of $T$ is of the form $(0,y_j)$ for any $j\in [1,l]$, then we get the contradiction that $S_1$ has an $A$-weighted zero-sum subsequence. If $T$ has a term of the form $(0,y_j)$ for some $j\in [1,l]$, then by taking the projection to the second coordinate, we get the contradiction that $S_2$ has a $B$-weighted zero-sum subsequence. Thus, we see that $S$ does not have any $A\times B$-weighted zero-sum subsequence. Hence, we get that $D_{A\times B}(M\times N)\geq D_A(M)+D_B(N)-1$. 
\end{proof}

By a similar argument we get the next result. 

\begin{thm}\label{sumd}
Let $M$ and $N$ be $R$-modules and $A\su R$. Then $M\times N$ is an $R$-module and we have $D_A(M\times N)\geq D_A(M)+D_A(N)-1$. 
\end{thm}

\section{$C_{U(n)}(n)$}

\begin{thm}\label{c2a}
$C(\Z_2^{^a})=2^a$.  
\end{thm}

\begin{proof}
We will prove that $C(\Z_2^{^a})\geq 2^a$ by using induction on $a$. This is clear when $a=1$. For $a>1$, by the induction hypothesis we have $C(\Z_2^{^{a-1}})\geq 2^{a-1}$. Taking $A=\{1\}$, $M=\Z_2^{^{a-1}}$ and $N=\Z_2$ in Theorem \ref{proc} gives us $C(\Z_2^{^a})\geq 2^a$. Also, from Theorem 1 of \cite{SKS} we get $C(\Z_2^{^a})\leq 2^a$. 
\end{proof}

\begin{thm}\label{cab}
 For any finite abelian group $G$, we have that $C(G)=|G|$. 
\end{thm}

\begin{proof}
By induction on the rank of $G$, Corollary 1 of \cite{SKS} and Theorem \ref{proc} we can show that $C(G)\geq |G|$. Also, from Theorem 1 of \cite{SKS} we get that $C(G)\leq |G|$.
\end{proof}

\begin{rem}
Theorem \ref{cab} is in contrast to Theorem 1 of \cite{O2} which says that for a finite abelian group $G$ we have that $D(G)=|G|$ if and only if $G$ is a finite cyclic group. 
\end{rem}

Let $x=(x_1,\ldots,x_m),y=(y_1,\ldots,y_m)\in \Z_2^{^m}$. We define the dot product of $x$ and $y$ to be $x\cdot y=x_1y_1+\ldots +x_my_m$.

\begin{lem}\label{gricv}
Let $v_1,\ldots,v_a\in \Z_2^{^m}$ where $m\geq 2^a$. Then there exists a non-zero vector $w\in \Z_2^{^m}$ such that for each $i\in [1,a]$ we have $w\cdot v_i=0$ and the coordinates of $w$ which are equal to one occur in consecutive positions.   
\end{lem}

\begin{proof}
Let $P$ be the matrix of size $m\times a$ whose columns are the vectors $v_i$. If $w\in \Z_2^{^m}$, then $wP$ is the vector in $\Z_2^{^a}$ which is the sum of those rows of the matrix $P$ which correspond to the coordinates of $w$ which are 1. From Theorem  \ref{c2a} we have $C(\Z_2^{^a})=2^a$. As $m\geq 2^a$, the sum of some consecutive rows of $P$ is zero. Thus, we can find a vector $w\in \Z_2^{^m}$ as in the statement of the lemma.  
\end{proof}

\begin{lem}\label{gric}
Let $X_1,\ldots,X_a$ be subsets of $[1,m]$ where $m\geq 2^a$. Then there exists $Y\su\{1,\ldots,m\}$ such that the elements of $Y$ are consecutive numbers and for each $i\in [1,a]$ we have $|Y\cap X_i|$ is even.   
\end{lem}

\begin{proof}
We identify a subset $A\su [1,m]$ with the vector $x_A\in\Z_2^{^m}$ whose $j^{th}$ coordinate is one if and only if $j\in A$. For any two subsets $A$ and $B$ of $[1,m]$ we observe that $x_A\cdot x_B=0$ if and only if $|A\cap B|$ is even. So we see that Lemma \ref{gric} follows from Lemma \ref{gricv}.  
\end{proof}

For a divisor $m$ of $n$, we define the natural map $f_{n,m}:\Z_n\to\Z_m$ to be the map given by $f_{n,m}\big(x+n\Z\big)=x+m\Z$. Let $p$ be a prime divisor of $n$ with $v_p(n)=r$. For a sequence $S$ in $\Z_n$, we denote its image under $f_{n,p^r}$ by $S^{(p)}$. 

\begin{obs}\label{obs}
A sequence $S$ is a $U(n)$-weighted zero-sum sequence in $\Z_n$ if and only if for every prime divisor $p$ of $n$ the sequence $S^{(p)}$ is a $U(p^r)$-weighted zero-sum sequence in $\Z_{p^r}$ where $v_p(n)=r$.
\end{obs}

This is Observation 2.2 in \cite{G}. The next result follows from Lemma 2.4 in \cite{G} along with the remark which follows it. 

\begin{lem}\label{gri}
Let $p$ be a prime, $n=p^r$ for some $r$ and $S=(x_1,\ldots,x_m)$ a sequence in $\Z_n$. Suppose for each $i\in [1,r]$, the size of the set $X_i$ is even where $X_i=\{j:x_j\not\equiv 0~(mod~p^i)\}$. Then $S$ is a $U(n)$-weighted zero-sum sequence. 
\end{lem}

\begin{thm}\label{cun}
$C_{U(n)}(n)=2^{\Omega(n)}$. 
\end{thm}

\begin{proof}
By Corollary 2 of \cite{SKS} we have $C_{U(n)}(n)\geq 2^{\Omega(n)}$. Let $a=\Omega(n)$ and $S=(x_1,\ldots,x_m)$ be a sequence in $\Z_n$ of length $m=2^a$. If we show that $S$ has a $U(n)$-weighted zero-sum subsequence of consecutive terms, it will follow that $C_{U(n)}(n)\leq 2^{\Omega(n)}$. For any prime divisor $p$ of $n$ and for each $i\in [1,v_p(n)]$ we let  
$$X_i^{(p)}=\{\,j:x_j\not\equiv 0~(mod~p^i)\,\}.$$ 
We observe that we have $a$ sets in the collection 
$$\big\{\,X_i^{(p)}:p~\textrm{is a prime divisor of}~n~\textrm{and}~i\in [1,v_p(n)]\,\big\}.$$ 
By Lemma \ref{gric} we have a subset $Y\su [1,m]$ such that all the elements of $Y$ are consecutive numbers and for each prime divisor $p$ of $n$ and for each $i\in [1,v_p(n)]$ we have $|\,Y\cap X_i^{(p)}\,|$ is even. Let $T$ be the subsequence of $S$ such that $x_j$ is a term of $T$ if and only if $j\in Y$.

Let $q$ be a prime divisor of $n$ and $v_q(n)=r$. By Lemma \ref{gri} the sequence $T^{(q)}$ is a $U(q^r)$-weighted zero-sum sequence in $\Z_{q^r}$ as for each $i\in [1,r]$ the set  
$$Y\cap X_i^{(q)}=\{\,j\in Y:x_j\not\equiv 0~(mod~q^i)\,\}$$ 
has even size. Thus, by Observation \ref{obs} it follows that the sequence $T$ is a $U(n)$-weighted zero-sum sequence. Hence, we see that $S$ has a $U(n)$-weighted zero-sum subsequence of consecutive terms.  
\end{proof}

\section{$D_{U(n)}(n)$}

The next result is a special case of the main result in \cite{O}. We are able to give a very simple argument and have hence included the proof. We can also prove Theorem \ref{d2a} by using Theorem \ref{sumd} and the easy observation that $D(\Z_2)=2$. 

\begin{thm}\label{d2a}
$D(\Z_2^{^a})=a+1$.  
\end{thm}

\begin{proof}
Let $B=\{x_1,\ldots,x_a\}$ be a basis of the $\Z_2$-vector space $\Z_2^{^a}$. As $B$ is linearly independent, the sequence $S=(x_1,\ldots,x_a)$ does not have any zero-sum subsequence and so we see that  $D(\Z_2^{^a})\geq a+1$. Any sequence of length $a+1$ in $\Z_2^{^a}$ has a zero-sum subsequence as any set of $a+1$ vectors is linearly dependent. Hence, we see that $D(\Z_2^{^a})\leq a+1$.
\end{proof}

\begin{lem}\label{griv}
Let $v_1,\ldots,v_a\in \Z_2^{^m}$ where $m\geq a+1$. Then there exists a non-zero vector $w\in \Z_2^{^m}$ such that for each $i\in [1,a]$ we have $w\cdot v_i=0$.   
\end{lem}

\begin{proof}
Let $P$ be the matrix of size $m\times a$ whose columns are the vectors $v_i$. If $w\in \Z_2^{^m}$, then $wP$ is the vector in $\Z_2^{^a}$ which is the sum of those rows of the matrix $P$ which correspond to the coordinates of $w$ which are 1. From Theorem \ref{d2a} we have $D(\Z_2^{^a})=a+1$. As $m\geq a+1$, the sum of some rows of $P$ is zero. Thus, we can find a vector $w\in \Z_2^{^m}$ as in the statement of the lemma.  
\end{proof}

The proof of the next result is similar to the proof of Lemma \ref{gric}.

\begin{lem}\label{grif}
Let $X_1,\ldots,X_a$ be subsets of $\{1,\ldots,m\}$, where $m\geq a+1$. Then there exists a non-empty subset $Y\su\{1,\ldots,m\}$ such that for each $i\in [1,a]$, we have $|\,Y\cap X_i\,|$ is even.   
\end{lem}

\begin{thm}\label{dun}
$D_{U(n)}(n)=\Omega(n)+1$. 
\end{thm}

\begin{proof}
For a prime $p$, the argument in the proof of Theorem 2 of \cite{SKS} shows that $D_{U(p)}(p)=2$. So from Lemma 1.8 of \cite{SKS3} we see that $D_{U(n)}(n)\geq\Omega(n)+1$.

Let $S=(x_1,\ldots,x_m)$ be a sequence in $\Z_n$ of length $m=\Omega(n)+1$. If we show that $S$ has a $U(n)$-weighted zero-sum subsequence, then it will follow that $D_{U(n)}(n)\leq \Omega(n)+1$. For any prime divisor $p$ of $n$ and for each $i\in [1,v_p(n)]$ let $X_i^{(p)}=\{\,j:x_j\not\equiv 0~(mod~p^i)\,\}$. The collection 
$$\{\,X_i^{(p)}:p~\textrm{is a prime divisor of}~n~\textrm{and}~i\in [1,v_p(n)]\,\}$$ 
has $\Omega(n)$ sets. As $m=\Omega(n)+1$, by Lemma \ref{grif} we have a non-empty subset $Y\su\{1,2,\ldots,m\}$ such that for each prime divisor $p$ of $n$ and for each $i\in [1,v_p(n)]$ we have $|\,Y\cap X_i^{(p)}\,|$ is even. Let $T$ be the subsequence of $S$ such that $x_j$ is a term of $T$ if and only if $j\in Y$.

Let $p$ be a prime divisor of $n$ and let $v_p(n)=r$. By Lemma \ref{gri} the sequence $T^{(p)}$ is a $U(p^r)$-weighted zero-sum sequence in $\Z_{p^r}$, as for each $i\in [1,r]$ the set  
$$Y\cap X_i^{(p)}=\{\,j\in Y:x_j\not\equiv 0~(mod~p^i)\,\}$$ 
has even size. Thus, by Observation \ref{obs} the sequence $T$ is a $U(n)$-weighted zero-sum sequence. Hence, it follows that $S$ has a $U(n)$-weighted zero-sum subsequence. 
\end{proof}

\section{Zero-sum subsequences of length $m$ in $\Z_2^{^a}$}

If $T$ is a subsequence of a sequence $S$, then $S-T$ denotes the sequence which is obtained by removing the terms of $T$ from $S$. 

\begin{obs}\label{zsn}
Let $(G,+)$ be a finite abelian group of order $n$ and $a\in G$. Let $S$ be a sequence of length $n$ in $G$ and $S-a$ denote the sequence which is obtained by subtracting $a$ from each term of $S$. If $S-a$ is a zero-sum sequence then $S$ is a zero-sum sequence. 
\end{obs}

We now rederive a weaker version of Corollary 3.2 of \cite{G}. The proof given in \cite{G} uses Theorem 1.1 of \cite{E}. We have given a different proof which uses the value of $D(\Z_2^{^a})$. 

\begin{thm}\label{e2a}
Let $S$ be a sequence in $\Z_2^{^a}$ of length $m+a$ where $m$ is even and $m\geq 2^a$. Then $S$ has a zero-sum subsequence of length $m$.  
\end{thm}

\begin{proof}
Let $m\geq 2^a$ and $S=(x_1,\ldots,x_k)$ be a sequence in $\Z_2^{^a}$ of length $k=m+a$.  Suppose each term of $S$ occurs an even number of times in $S$. As $x+x=0$ for any $x\in \Z_2^{^a}$, it follows that $S$ has a zero-sum subsequence of any even length and hence also of length $m$. So we may assume that some term of $S$ occurs an odd number of times. By Observation \ref{zsn} we can assume that that term is $0$.

Let $S'$ be the unique subsequence of $S$ whose  terms are all distinct and such that each term of $S-S'$ occurs an even number of times in $S-S'$.  
Let the length of $S'$ be $k'$. Suppose $k'\leq a$. It follows that $k-k'=(m+a)-k'$ and so $m\leq k-k'$. As $m$ is even, we see that $S-S'$ (and hence $S$) has a zero-sum subsequence of length $m$.

So we may assume that $k'\geq a+1$. By Theorem \ref{d2a} we get $D(\Z_2^{^a})=a+1$ and so we see that $S'$ has a zero-sum subsequence. Let $T$ be a zero-sum subsequence of $S'$ having  largest length. As $0$ occurs an odd number of times in $S$, it is a term of $S'$ and so $0$ is a term of $T$. Let the length of $T$ be $l$. As $D(\Z_2^{^a})=a+1$ and $S'-T$ does not have any zero-sum subsequence, it follows that $k'-l\leq a$. We now claim that 
\begin{equation}\label{ea}
l\leq m\leq (k-k')+l.
\end{equation} 
As all terms of $S'$ are distinct, it follows that $k'\leq 2^a$. So as $l\leq k'$ and $2^a\leq m$, we see that $l\leq m$. As $k'-l\leq a$ and $m=k-a$, we have that $m\leq (k-k')+l$. This proves our claim (\ref{ea}).

Suppose $l$ is even. From (\ref{ea}) we see that $m-l\leq k-k'$.  As $m-l$ is even and each term of $S-S'$ occurs an even number of times, it follows that $S-S'$ has a zero-sum subsequence having length $m-l$. As $T$ is a zero-sum subsequence of $S'$ having length $l$, we get a zero-sum subsequence of $S$ having length $m$.

Suppose $l$ is odd. As both $m$ and $k-k'$ are even, from (\ref{ea}) we see that $m\leq (k-k')+l-1$. As $m-(l-1)$ is even and each term of $S-S'$ occurs an even number of times, it follows that $S-S'$ has a zero-sum subsequence having length $m-(l-1)$. As $0$ is a term of $T$ which has length $l$, we get a zero-sum subsequence of $S'$ having length $l-1$. So we get a zero-sum subsequence of $S$ having length $m$. 
\end{proof}

\begin{lem}
Let $v_1,\ldots,v_a\in \Z_2^{^{m+a}}$ where $m$ is even and $m\geq 2^a$. Then there exists a non-zero vector $w\in \Z_2^{^{m+a}}$ such that $w$ has exactly $m$ terms which are one and for $i\in [1,a]$ we have $w\cdot v_i=0$.     
\end{lem}

\begin{proof}
The proof of this lemma is similar to the proof of Lemma \ref{griv}, where we use Theorem \ref{e2a} in place of Theorem \ref{d2a}. 
\end{proof}

The next result follows from this lemma. 

\begin{lem}\label{grie}
Let $X_1,\ldots,X_a$ be subsets of $\{1,\ldots,m+a\}$ where $m$ is even and $m\geq 2^a$. Then there exists a non-empty subset $Y\su\{1,\ldots,m+a\}$ such that $|Y|=m$ and for each $i\in [1,a]$ we have $|\,Y\cap X_i\,|$ is even.   
\end{lem}

\begin{thm}\label{eun}
Let $a=\Omega(n)$ and $m$ be an even number such that $m\geq 2^a$. Suppose $S$ is a sequence in $\Z_n$ having length $m+a$. Then $S$ has a $U(n)$-weighted zero-sum subsequence of length $m$. 
\end{thm}

\begin{proof}
The proof of this theorem is similar to the proof of Theorem \ref{dun}, where we use Lemma \ref{grie} in place of Lemma \ref{grif}. 
\end{proof}

\begin{rem}
Theorem 1.3 of \cite{G} is a generalization of Theorem \ref{eun}. The argument which is given in \cite{G} for even $n$, uses a weaker version of Theorem 1.1 of \cite{E}. We give a different argument in the proof of Theorem \ref{eun} which is more direct.
\end{rem}


\begin{cor}\label{code}
Let $m$ be even and $m\geq 2^a$. Then every $m$-dimensional subspace of $\Z_2^{^{m+a}}$ has a vector $v$ such that exactly $m$ coordinates of $v$ are equal to one.  
\end{cor}

\begin{proof}
Let $W$ be an $m$-dimensional subspace of $\Z_2^{^{m+a}}$. Then the dimension of $W^\perp$ is $a$. Let $\{v_1,\ldots,v_a\}$ be a basis of $W^\perp$ and $A$ be the $a\times (m+a)$ matrix whose rows are the basis vectors of $W^\perp$. Then $A$ gives a map $\Z_2^{^{m+a}}\to\Z_2^{^a}$ whose kernel is $(W^\perp)^\perp$ which can be shown to be $W$. For each $j\in [1,m+a]$, let $C_j$ denote the $j^{th}$ column of $A$. By Theorem \ref{e2a} we can find a subset $I\su [1,m+a]$ such that $|I|=m$ and $\sum_{i\in I}C_i=0$. Let $v=(x_1,\ldots,x_{m+a})\in \Z_2^{^{m+a}}$ be such that $x_i=1$ if and only if $i\in I$. As we have that $Av=\sum_{i\in I}C_i$, it follows that $v\in ker~A=W$. Hence, we see that $W$ has a vector $v$ such that exactly $m$ coordinates of $v$ are equal to one.  
\end{proof}

\begin{defn}
For a vector $v\in\Z_2^{^n}$ we define the  weight of $v$ to be the number of coordinates of $v$ which are equal to one.
Let $l(n,\bar m)$ denote the largest integer $k$ such that there is a subspace of $\Z_2^{^n}$ of dimension $k$ which does not have any vector of weight $m$. The notation $l(n,\bar m)$ was introduced in \cite{E}. 
\end{defn}

\begin{cor}\label{eno}
Let $m$ and $n$ be such that $m$ be even, $m>n$ and $m\geq 2^{n-m}$. Then we have that $l(n,\bar m)=m-1$.
\end{cor}

\begin{proof} 
For each $i\in [1,n]$ let $e_i$ denote the vector in $\Z_2^{^n}$ whose only non-zero coordinate is one which occurs in the $i^{th}$ position. The subspace of $\Z_2^{^n}$ which is spanned by $\{e_1,\ldots,e_{m-1}\}$ does not have any vector of weight $m$ and hence it follows that $l(n,\bar m)\geq m-1$. 
So by using Corollary \ref{code} we see that if $m$ is even, $m>n$ and $m\geq 2^{n-m}$, we have that $l(n,\bar m)=m-1$.
\end{proof}

Corollary \ref{eno} is a weaker version of Theorem 1.1 of \cite{E}. We have rederived the result in Corollary \ref{eno} here as it is an easy consequence of Theorem \ref{e2a}.

\section{Extremal sequences for $U(n)$ where $n=2^k$}

\begin{defn}
Let $A$ be a subgroup of $U(n)$. Suppose $S=(x_1,\ldots,x_k)$ and $T=(y_1,\ldots,y_k)$ are sequences in $\Z_n$. We say that $S$ and $T$ are $A$-equivalent if there is a unit $c\in U(n)$, a permutation $\sigma\in S_k$ and $a_1,\ldots,a_k\in A$ such that for each $i\in [1,k]$ we have $c\,y_{\sigma(i)}=a_ix_i$. 
\end{defn}

\begin{rem}
Let $A$ be a subgroup of $U(n)$. If $S$ is an $E$-extremal sequence (a $D$-extremal sequence) for $A$ and if $S$ and $T$ are $A$-equivalent, then $T$ is also an $E$-extremal sequence (a $D$-extremal sequence) for $A$. 
\end{rem}

When $n=p^r$ where $p$ is an odd prime, in Theorem 3 of \cite{AMP} it was shown that a sequence in $\Z_n$ is an $E$-extremal sequence for $U(n)$ if and only if it is $U(n)$-equivalent to the sequence $$(\underbrace{0,0,\ldots,0}_\text{$(n-1)$ times},1,p,p^2,\ldots,p^{r-1}).$$

In Theorem 4 of \cite{AMP} it was shown that when $n=2^r$, a sequence in $\Z_n$ is a $D$-extremal sequence for $U(n)$ if and only if it is $U(n)$-equivalent to the sequence $(1,2,2^2,\ldots,2^{r-1})$.

The next result has been shown in \cite{G} and \cite{L}. We can also derive this as a consequence of Theorem \ref{eun}. 

\begin{thm}\label{en}
For any $n$ we have $E_{U(n)}(n)=n+\Omega(n)$. 
\end{thm}

\begin{thm}
A sequence in $\Z_{2^r}$ is an $E$-extremal sequence for $U(2^r)$ if it is $U(2^r)$-equivalent to a sequence of length $n+r-1$ in which $2^i$ occurs exactly once for each $i\in [0,r-2]$ and there exists an odd number $m\in [1,2^r-1]$ such that $2^{r-1}$ occurs exactly $m$ times  and the remaining terms are zero. 
\end{thm}

\begin{proof}
Let $S$ be a sequence as in the statement of the theorem. Suppose $T$ is a $U(2^r)$-weighted zero-sum subsequence of $S$ of length $2^r$. Consider the set $J=\{\,i\in [0,r-2]:2^i\in T\,\}$. Suppose $J\neq \emptyset$. Then $J$ has a least element $i_0$. As $T$ cannot have only one non-zero term, we see that $2^i$ is a term of $T$ for some $i\in [i_0+1,r-1]$. As $2^{i_0+1}$ divides all the terms of $T$ except $2^{i_0}$ and as $T$ is a $U(2^r)$-weighted zero-sum sequence, we get the contradiction that a unit is divisible by 2.

Thus we see that $J=\emptyset$. It follows that $T$ is a sequence of length $2^r$ in which there are an odd number of non-zero terms  all of which are equal to $2^{r-1}$. As $T$ is a $U(2^r)$-weighted zero-sum sequence, we see that an odd multiple of $2^{r-1}$ is zero. This gives the contradiction that $2^{r-1}=0$. Hence, we see that $S$ is a sequence of length $2^r+r-1$ which does not have any $U(2^r)$-weighted zero-sum subsequence of length $2^r$. From Theorem \ref{en} we get that $E_{U(2^r)}(2^r)=2^r+r$. So we see that $S$ is an $E$-extremal sequence for $U(2^r)$. It follows that a sequence which is $U(2^r)$-equivalent to $S$ is also an $E$-extremal sequence for $U(2^r)$. 
\end{proof}

The next result is Lemma 1 (ii) of \cite{L}. 

\begin{lem}\label{even}
If a sequence in $\Z_{2^r}$ has a non-zero even number of units, then it is a $U(2^r)$-weighted zero-sum sequence.  
\end{lem}

\begin{thm}
If a sequence in $\Z_{2^r}$ is an $E$-extremal sequence for $U(2^r)$, then it is $U(2^r)$-equivalent to a sequence in which $2^i$ occurs exactly once for each $i\in [0,r-2]$ and there is an odd number $m\in [1,2^r-1]$ such that $2^{r-1}$ occurs exactly $m$ times and the remaining terms are zero. 
\end{thm}

\begin{proof}
Let $n=2^r$ and $S$ be a sequence in $\Z_n$ which is an $E$-extremal sequence for $U(n)$. By Theorem \ref{en} we have that $E_{U(n)}(n)=n+r$. So any $E$-extremal sequence for $U(n)$ has length $n+r-1$. Suppose $S$ has at least two units. By Lemma \ref{even} we see that $S$ has a zero-sum subsequence of length $t$ for any even $t$ which is at most $k-1$. We may assume that $r\geq 2$ and so we get that $k-1=n+r-2\geq n$. Thus, we get the contradiction that $S$ has a $U(n)$-weighted zero-sum subsequence of length $n$. Hence, we see that $S$ has at most one unit.

Let $s\leq r-2$. Suppose for each $i\in [0,s-1]$, the sequence $S$ has at most one term which is a unit multiple of $2^i$. We claim that $S$ has at most one term which is a unit multiple of $2^s$. By our assumption, we see that $S$ has at least $k-s$ terms which are divisible by $2^s$. If our claim is not true, we can find a subsequence of $S$ having length $k-s-1$ which has an even number of terms which are a unit multiple of $2^s$. 

So given any even number $t$ which is at most $k-s-1$, by Lemma \ref{even} we see that $S$ has a $U(n)$-weighted zero-sum subsequence having length $t$. As $n=k-(r-1)$ and $s\leq r-2$, it follows that $n\leq k-1-s$. Thus, we get the contradiction that $S$ has a $U(n)$-weighted zero-sum subsequence of length $n$. Hence, our claim must be true. Thus, we see by induction that for each $i\in [0,r-2]$ the sequence $S$ can have at most one term which is a unit multiple of $2^i$.

We now claim that for each $i\in [0,r-2]$ the sequence $S$ has exactly one term which is a unit multiple of $2^i$. If not, there are at most $r-2$ such terms and so $S$ will have at least $k-(r-2)=n+1$ terms which are either zero or a unit multiple of $2^{r-1}$. We can find a subsequence of $S$ having length $n$ which has an even number of terms which are a unit multiple of $2^{r-1}$. So by Lemma \ref{even} we get the contradiction that $S$ has a $U(n)$-weighted zero-sum subsequence of length $n$. Hence, we see that our claim is true.

Thus, we see that $S$ has $k-(r-1)=n$ terms which are either zero or a unit multiple of $2^{r-1}$. By Lemma \ref{even} we see that the number of terms of $S$ which are a unit multiple of $2^{r-1}$ must be odd. Thus, $S$ is $U(n)$-equivalent to a sequence in which $2^i$ occurs exactly once for each $i\in [0,r-2]$ and there is an odd number $m\in [1,n-1]$ such that $2^{r-1}$ occurs exactly $m$ times and the remaining terms are zero. 
\end{proof}

For example, a sequence in $\Z_8$ is an $E$-extremal sequence for $U(8)$ if and only if it is $U(8)$-equivalent to one of the following sequences: 
$$(\,1,\,2,\,4,\,0,\,0,\,0,\,0,\,0,\,0,\,0\,),~(\,1,\,2,\,4,\,4,\,4,\,0,\,0,\,0,\,0,\,0\,),$$
$$(\,1,\,2,\,4,\,4,\,4,\,4,\,4,\,0,\,0,\,0\,),~(\,1,\,2,\,4,\,4,\,4,\,4,\,4,\,4,\,4,\,0\,).$$ 

We will now characterize the $C$-extremal sequences for $U(2^r)$. The following result is Lemma 5 in \cite{SKS}.

\begin{lem}\label{b}
Let $S$ be a sequence in $\Z_n$ and $p$ be a prime divisor of $n$ such that every term of $S$ is divisible by $p$. Suppose $n'=n/p$ and $S'$ is the sequence in $\Z_{n'}$ whose terms are obtained by dividing the terms of $S$ by $p$ and taking their images under $f_{n,n'}$. If $S'$ is a $U(n')$-weighted zero-sum sequence, then $S$ is a $U(n)$-weighted zero-sum sequence.  
\end{lem}

\begin{thm}
Let $n=2^r$, $n'=n/2$ and $S=(x_1,\ldots,x_{n-1})$ be a sequence in $\Z_n$. Suppose $S_1=(x_1,\ldots,x_{n'-1})$ and $S_2=(x_{n'+1},\ldots,x_{n-1})$. Then $S$ is a $C$-extremal sequence for $U(n)$ if and only if all the terms of $S$ are even except the `middle' term $x_{n'}$ and $S_1',\,S_2'$ are $C$-extremal sequences for $U(n')$ in $\Z_{n'}$ where  $S_1'$ and $S_2'$ denote the sequences in $\Z_{n'}$ which are obtained by dividing the terms of $S_1$ and $S_2$ by two and then taking their images under $f_{n,n'}$.
\end{thm}

\begin{proof}
Let $n=2^r$, $n'=n/2$ and $S=(x_1,\ldots,x_{n-1})$ be a $C$-extremal sequence for $U(n)$. Suppose $S$ has at least two odd terms. We can find a subsequence $T$ of consecutive terms of $S$ which has exactly two odd terms. By Lemma \ref{even} we get the contradiction that $T$ is a $U(n)$-weighted zero-sum sequence. Thus, we see that at most one term of $S$ can be odd.

Suppose the `middle' term $x_{n'}$ is even. As at most one term of $S$ is odd, we can find a subsequence $T$ of consecutive terms of $S$ of length $n'$ all of whose terms are even. Let $T'$ be the sequence in $\Z_{n'}$ whose terms are obtained by dividing the terms of $T$ by two and then taking their images under $f_{n,n'}$. As $n'=2^{r-1}$ by Theorem \ref{cun} we have that $C_{U(n')}(n')=n'$. Hence, as $T'$ has length $n'$ we see that $T'$ has an $U(n')$-weighted zero-sum subsequence of consecutive terms. So by Lemma \ref{b} we get the contradiction that $T$ (and hence $S$) has a $U(n)$-weighted zero-sum subsequence of consecutive terms.

Thus, we see that $S$ has a unique odd term which is $x_{n'}$. For $i=1,2$ consider the sequences $S_i$ and $S_i'$ which are as defined in the statement of the theorem. Suppose $S_1'$ has a $U(n')$-weighted zero-sum subsequence of consecutive terms. By Lemma \ref{b} we get the contradiction that $S_1$ (and hence $S$) has a $U(n)$-weighted zero-sum subsequence of consecutive terms. As $S_1'$ has length $n'-1$, it follows that $S_1'$ is a $C$-extremal sequence for $U(n')$ in $\Z_{n'}$. A similar argument shows that $S_2'$ is also a $C$-extremal sequence for $U(n')$ in $\Z_{n'}$.

The proof of the converse part is similar to the proof of Theorem 5 of \cite{SKS2} and so we will omit it.
\end{proof}

For example, a sequence $S$ in $\Z_8$ is a $C$-extremal sequence for $U(8)$ if and only if  there exist units $a_1,\ldots,a_7\in \Z_8$ such that $$S=(4a_1,2a_2,4a_3,a_4,4a_5,2a_6,4a_7).$$

\section{Some other weight-sets}

For a real number $x$, we denote the smallest integer which is greater than or equal to $x$ by $\lceil x\rceil$. When $A=\{1,2,\ldots,r\}$ where $r<n$, it is shown in Theorem 3 (i) of \cite{X} that $D_A(n)=\lceil n/r \rceil$. As we have that $C_A(n)\geq D_A(n)$, it follows that $C_A(n)\geq \lceil n/r \rceil$. 

\begin{thm}
Let $A=\{1,2,\ldots,r\}$ where $r<n$. Then we get $C_A(n)=\lceil n/r \rceil$. 
\end{thm}

\begin{proof}
It suffices to show that $C_A(n)\leq\lceil n/r \rceil$. Let $S=(x_1,\ldots,x_m)$ be a sequence in $\Z_n$ of length $m=\lceil n/r \rceil$. Consider the sequence $$S'=(\,\overbrace{x_1,\,\ldots,\,x_1}^\text{$r$ times},\,\overbrace{x_2,\,\ldots,\,x_2}^\text{$r$ times},\,\ldots,\,\overbrace{x_m,\,\ldots,\,x_m}^\text{$r$ times}\,).$$  
We observe that the length of $S'$ is $mr$ which is at least $n$. By Theorem 1 of \cite{SKS} we have $C_{\{1\}}(n)\leq n$ and so we see that $S'$ has a zero-sum subsequence of consecutive terms. Thus, we obtain an $A$-weighted zero-sum subsequence of $S$ having consecutive terms. Hence, it follows that $C_A(n)\leq\lceil n/r \rceil$. \end{proof}


\begin{thm}\label{odd}
 Let $n$ be even, $v_2(n)=r$ and $m=n/2^r$. Let $B$ denote the set of all odd elements of $\Z_n$ and $\{m\}\su A\su B$. Then we have that $C_A(n)=2^r$. 
\end{thm}

\begin{proof}
 Let $S=(x_1,\ldots,x_k)$ be a sequence in $\Z_n$ having length $k=2^r$. Let $r=v_2(n)$, $m=n/2^r$ and $S'$ denote the image of the sequence $S$ under the map $f_{n,2^r}$. By Theorem 1 of \cite{SKS} we have that $C_{\{1\}}(2^r)\leq 2^r$. So there is a subsequence $T$ of $S$ having consecutive terms such that the image $T'$ of $T$ under $f_{n,2^r}$ is a zero-sum subsequence. It follows that the sum of the terms of $T$ is divisible by $2^r$. So we see that $T$ is an $\{m\}$-weighted zero-sum sequence. Thus, we see that $C_{\{m\}}(n)\leq 2^r$. 
 
 By Corollary 2 of \cite{SKS} we have that $C_{U(2^r)}(2^r)\geq 2^r$. So there is a sequence $S'$ of length $2^r-1$ in $\Z_{2^r}$ which has no $U(2^r)$-weighted zero-sum subsequence of consecutive terms. As the map $f_{n,2^r}$ is onto, we can find a sequence $S$ in $\Z_n$ whose image under $f_{n,2^r}$ is $S'$. Suppose $S$ has a $B$-weighted zero-sum subsequence having consecutive terms. As $n$ is even, we see that the image of $B$ under $f_{n,2^r}$ is contained in $U(2^r)$. So we get the contradiction that $S'$ has a $U(2^r)$-weighted zero-sum subsequence having consecutive terms. Thus, we see that $C_{B}(n)\geq 2^r$. 
 
 As we have that $\{m\}\su A\su B$, it follows that $C_B(n)\leq C_A(n)\leq C_{\{m\}}(n)$ and so we get that $C_A(n)=2^r$. 
 \end{proof}

Our next result generalizes a result of \cite{X}. We follow a similar argument as in Theorem \ref{odd}.
 
\begin{thm}
Let $n$ be even, $v_2(n)=r$ and $m=n/2^r$. Let $B$ denote the set of all odd elements of $\Z_n$ and $\{m\}\su A\su B$. Then we have that $D_A(n)=r+1$. 
\end{thm}

\begin{proof}
Let $S=(x_1,\ldots,x_k)$ be a sequence in $\Z_n$ having length $k=r+1$. Let $r=v_2(n)$, $m=n/2^r$ and $S'$ denote the image of the sequence $S$ under the map $f_{n,2^r}$. By Theorem \ref{dun} we have that $D_{U(2^r)}(2^r)=r+1$. So there is a subsequence $T$ of $S$ such that the image $T'$ of $T$ under $f_{n,2^r}$ is a $U(2^r)$-weighted  zero-sum subsequence. By Lemma 7 of \cite{SKS2} we see that  the map $f_{n,2^r}$ maps $U(n)$ onto $U(2^r)$. So if $I$ denotes the set $\{i:x_i$ is a term of $T\}$, we see that for each $i\in I$ there exists $a_i\in U(n)$ such that $f_{n,2^r}\big(\sum_{i\in I} a_ix_i\big)=0$. Hence, it follows that $\sum_{i\in I} a_ix_i$ is divisible by $2^r$ and so we see that $T$ is an $\{m\}$-weighted zero-sum sequence. Thus, we see that $D_{\{m\}}(n)\leq r+1$. 

By Theorem \ref{dun} we have that $D_{U(2^r)}(2^r)=r+1$. So there is a sequence $S'$ of length $r$ in $\Z_{2^r}$ which has no $U(2^r)$-weighted zero-sum subsequence. The rest of the proof is very similar to the argument given in the second and third  paragraphs of the proof of Theorem \ref{odd}. Hence, we get that $D_A(n)=r+1$. 
\end{proof}

\begin{rem}\label{half-odd}
When $n$ is even and $A$ is the set of all even elements of $\Z_n$, in \cite{X} it is shown that $D_A(n)=2$ and so it follows that $C_A(n)=2$. 

When $n$ is odd and $A$ is the set of all odd (or all even) elements of $\Z_n$, in \cite{X} it is shown that $D_A(n)=3$. From the proof of this result we see that $C_A(n)=3$. 
\end{rem}

\begin{thm}\label{cubp}
Let $p$ and $k$ be odd primes such that $p\equiv 1~(mod~k)$ and $p\not\equiv 1~(mod~k^2)$. Then we have that $D_{U(p)^k}(p)\leq k$. 
\end{thm}

\begin{proof}
As $p\equiv 1~(mod~k)$ there exists $c\in U(p)$ such that $c$ has order $k$. The subgroup $U(p)^k$ is the image of the map $U(p)\to U(p)$ given by $x\mapsto x^k$. The kernel of this map has at most $k$ elements and contains $c$. As $c$ has order $k$, it follows that the kernel is $\langle\, c \,\rangle$. Thus, we have that $|\,U(p)\,|=k\,|U(p)^k|$ and so we see that $U(p)/U(p)^k$ has order $k$.

Suppose $c\in U(p)^k$. Then there exists $a\in U(p)$ such that $c=a^k$ and so $a^{k^2}=c^k=1$. As $k$ is a prime and $a^k=c\neq 1$, we see that the order of $a$ is $k^2$. However, as $p\not\equiv 1~(mod~k^2)$ there is no element of order $k^2$ in $U(p)$. Thus, we see that $c\notin U(p)^k$. For $x\in U(p)$ if we denote the coset $xU(p)^k$ by $[x]$ we see that $[c]\neq [1]$. As $k$ is prime we see that $[c]$ has order $k$. Hence, we get a partition of $U(p)$ by the cosets $[1],\,[c],\,\ldots,\,[c^{k-1}]$.

Let $S=(x_1,\ldots,x_k)$ be a sequence in $U(p)$. Suppose we show that $S$ has a $U(p)^k$-weighted zero-sum subsequence. It will follow that $D_{U(p)^k}(p)\leq k$.

\begin{case}
There exist two elements of $S$ which are in the same coset.
\end{case}

As there exists $l\in [0,k-1]$ such that $x_i,x_j\in [c^{\,l}]$, it follows that there exist $a,b\in U(p)^k$ such that $x_i=a\,c^{\,l}$ and $x_j=b\,c^{\,l}$. Then we have that $(-b)\,x_i+a\,x_j=0$. As $k$ is odd, we see that $-1\in U(p)^k$. Hence, it follows that $(x_i,x_j)$ is a $U(p)^k$-weighted zero-sum subsequence of $S$.

\begin{case}
No two elements of $S$ are in the same coset.
\end{case}

Without loss of generality, we can assume that for each $i\in [1,k]$ there exist $a_i\in U(p)^k$ such that $x_i=a_i\,c^{\,i-1}$. Then $a_1^{-1}\,x_1+a_2^{-1}\,x_2+\ldots+a_k^{-1}\,x_k=1+c+\ldots+c^{k-1}=0$ as $c$ satisfies $X^k-1=(X-1)\,(X^{k-1}+\ldots+X+1)$ and $c\neq 1$. Thus, it follows that $S$ is a $U(p)^k$-weighted zero-sum sequence. 
\end{proof}

From Theorem 3 of \cite{AR} we see that $D_{U(p)^2}(p)=3$ where $p$ is an odd prime. This shows that Theorem \ref{cubp} is not true when $k=2$. The next result is Corollary 1 of \cite{SKS2}.

\begin{lem}
Let $F$ be a field and $A$ be a subgroup of $F^*$. A sequence $S=(x,y)$ in $F$ does not have an $A$-weighted zero-sum subsequence if and only if $x$ and $-y$ are in different cosets of $A$ in $F^*$. 
\end{lem}

\begin{cor}\label{field}
 Let $F$ be a field and $A$ be a proper subgroup of $F^*$. Then we have that $D_A(F)\geq 3$. 
\end{cor}

\begin{cor}
Let $k\geq 2$ and $p$ be a prime such that $p\equiv 1~(mod~k)$. Then we have that $D_{U(p)^k}(p)\geq 3$.  
\end{cor}

\begin{proof}
As $p\equiv 1~(mod~k)$, there is an element in $U(p)$ of order $k$. So as in the first paragraph of the proof of Theorem \ref{cubp} we see that the index of the subgroup $U(p)^k$ of $U(p)$ is $k$. As $k\geq 2$ we see that $U(p)^k$ is a proper subgroup of $U(p)$ and so the result follows from Corollary \ref{field}.
\end{proof}

\section{Concluding remarks}

It will be interesting to characterize the extremal sequences for $U(n)$ for the $U(n)$-weighted zero-sum constants $D_{U(n)}(n)$, $C_{U(n)}(n)$ and $E_{U(n)}(n)$, when $n$ is an even number which is not a power of 2. In \cite{AMP} the $D$-extremal sequences for $U(n)$ have been characterized when $n$ is odd and the $E$-extremal sequences for $U(n)$ have been characterized when $n$ is a power of an odd prime. In \cite{SKS2} the $C$-extremal sequences for $U(n)$ have been characterized when $n$ is odd. 

In Theorem 3 of \cite{X} and Remark \ref{half-odd} it is shown that $C_A(n)=D_A(n)=3$ when $n$ is odd and $A$ is the set of all odd elements of $\Z_n$. In Theorem 3 of \cite{AR} and Theorem 4 of \cite{SKS} it is shown that $C_{Q_p}(p)=D_{Q_p}(p)=3$ where $p$ is an odd prime. This makes us curious to see if we can find a weight-set $A\su\Z_n$ where $n$ is odd such that $A$ has size $(n-1)/2$ and  we have that $D_A(n)>3$ or $C_A(n)>3$. 

\bigskip

{\bf Acknowledgement.}
Santanu Mondal would like to acknowledge CSIR, Govt. of India, for a research fellowship.


\begin{thebibliography}{12}
\bibitem{A}
S. D. Adhikari, The relation between two combinatorial group invariants: History and ramifications {\it Math. Student} {\bf 77} nos. 1-4 (2008), 131-143. 

\bibitem{AMP}
S. D. Adhikari, I. Molla and S. Paul, Extremal sequences for some weighted zero-sum constants for cyclic groups, {\it CANT IV, Springer Proc. Math. Stat.} {\bf 347} (2021), 1-10.

\bibitem{AR}
S. D. Adhikari and P. Rath, Davenport constant with weights and some related questions, {\it Integers} {\bf 6} (2006), \#A30. 

\bibitem{E}
H. Enomoto, P. Frankl, N. Ito and K. Nomura, Codes with given distances, {\it Graphs Combin.} {\bf 3} (1987), 25-38.

\bibitem{Gao}
W. D. Gao, A combinatorial problem on finite abelian groups, {\it J. Number Theory} {\bf 58} (1996), 100-103.

\bibitem{G}
S. Griffiths, The Erd\H{o}s-Ginzberg-Ziv theorem with units, {\it Discrete Math.} {\bf 308} no. 23 (2008), 5473-5484.

\bibitem{L}
F. Luca, A generalization of a classical zero-sum problem, {\it Discrete Math.} {\bf 307} no. 13 (2007), 1672-1678.

\bibitem{SKS} S. Mondal, K. Paul and S. Paul, On a different weighted zero-sum constant, {\it Discrete Math.} {\bf 346} no. 6 (2023), 113350.

\bibitem{SKS2} S. Mondal, K. Paul and S. Paul, Extremal sequences for a weighted zero-sum constant, {\it Integers} {\bf 22} (2022), \#A93.

\bibitem{SKS3} S. Mondal, K. Paul and S. Paul, Zero-sum constants related to the Jacobi symbol. arxiv:2111.14477v3.

\bibitem{O}
J.E. Olson, A combinatorial problem on finite abelian groups, I, {\it J. Number Theory} {\bf 1} (1969), 8-10. 

\bibitem{O2}
J.E. Olson, A combinatorial problem on finite abelian groups, II, {\it J. Number Theory} {\bf 1} (1969), 195-199.

\bibitem{X} 
X. Xia and Z. Li, Some Davenport constants with  weights and Adhikari and Rath’s conjecture, {\it Ars Combin.} {\bf 88} (2008), 83-95.

\bibitem{YZ}
P. Yuan and X. Zeng, Davenport constant with weights, {\it European J. Combin.} {\bf 31} (2010), 677-680.
\end{thebibliography}
\end{document}